\documentclass[12pt]{article}

\usepackage{amsmath,amsthm}
\usepackage{amssymb}
\usepackage{color}
\usepackage{breqn}
\usepackage{enumerate}
\usepackage{graphicx}
\usepackage{bbm}
\usepackage[affil-it]{authblk}

\newtheorem{theorem}{Theorem}[section]
\newtheorem{corollary}[theorem]{Corollary}
\newtheorem{lemma}[theorem]{Lemma}
\newtheorem{proposition}[theorem]{Proposition}
\newtheorem{question}[theorem]{Question}

\newtheorem{observation}[theorem]{Observation}
\theoremstyle{definition}

\makeatletter
\def\blfootnote{\gdef\@thefnmark{}\@footnotetext}
\makeatother	
	
\begin{document}
\date{}
\title{Sets in Almost General Position}
\author{Luka Mili\'{c}evi\'{c}
\thanks{E-mail address: \texttt{lm497@cam.ac.uk}}}
\affil{Department of Pure Mathematics and Mathematical Statistics\\Wilberforce
Road\\Cambridge CB3 0WB\\UK}

\maketitle
\abstract{Erd\H{o}s asked the following question: given $n$ points in the plane in almost general position (no 4 collinear), how large a set can we guarantee to find that is in general position (no 3 collinear)? F\"{u}redi constructed a set of $n$ points in almost general position with no more than $o(n)$ points in general position. Cardinal, T\'{o}th and Wood extended this result to $\mathbb{R}^3$, finding sets of $n$ points with no 5 on a plane whose subsets with no 4 points on a plane have size $o(n)$, and asked the question for higher dimensions: for given $n$, is it still true that the largest subset in general position we can guarantee to find has size $o(n)$? We answer their question for all $d$ and derive improved bounds for certain dimensions.}

\blfootnote{\textup{2010} \textit{Mathematics Subject Classification}: 52C35; 52C10}

\section{Introduction}
A set of points in the plane is said to be \emph{in general position} if it has no 3 collinear points, and \emph{in almost general position} if there are no 4 collinear points. Let $\alpha(n)$ be the maximum $k$ such that any set of $n$ points in the plane in almost general position has $k$ points in general position. 
In~\cite{Erdos}, Erd\H{o}s asked for an improvement of the (easy) bounds $\sqrt{2n-1} \leq \alpha(n) \leq n$ (see equation (13) in the paper). This was done by F\"{u}redi~\cite{Furedi}, who proved $\Omega(\sqrt{n\log n}) \leq \alpha(n) \leq o(n)$. 

In~\cite{CardinalTothWood} Cardinal, T\'{o}th and Wood considered the problem in $\mathbb{R}^3$. Firstly, let us generalize the notion of general position. A set of points in $\mathbb{R}^d$ is said to be in general position if there are no $d+1$ points on the same hyperplane, and in almost general position if there are no $d+2$ points on the same hyperplane. Let $\alpha(n,d)$ stand for the maximum integer $k$ such that all sets of $n$ points in $\mathbb{R}^d$ in almost general position contain subset of $k$ points in general position. Cardinal, T\'{o}th and Wood proved that $\alpha(n,3) = o(n)$ holds. They noted that for fixed $d \geq 4$, only $\alpha(n,d) \leq C n$ is known, for a constant $C$, and they asked whether $\alpha(n,d) = o(n)$. The goal of this paper is to answer their question in all dimensions. In particular we prove the following.

\begin{theorem}For a fixed integer $d \geq 2$, we have $\alpha(n,d) = o(n)$.\end{theorem}

In fact, we are able to get better bounds for certain dimensions. This is the content on the next theorem.

\begin{theorem}Suppose that $d,m \in \mathbb{N}$ satisfy $2^{m+1} - 1 \leq d \leq 3.2^m - 3$. Let $N \geq 1$. Then $$\alpha(2^N,d) \leq \left(\frac{25}{N}\right)^{1/2^{m+1}} 2^N.$$\end{theorem} 

It is worth noting the lower bound $\alpha(n,d) = \Omega_d((n\log n)^{1/d})$ due to Cardinal, T\'{o}th and Wood~\cite{CardinalTothWood}, but we do not try to improve their bound in this paper.\\
 
In~\cite{Furedi} F\"{u}redi used the density Hales-Jewett theorem (\cite{FurstenbergKatznelson},~\cite{Polymath}) to establish $\alpha(n) = \alpha(n,2) = o(n)$. Here we reproduce his argument. By the density Hales-Jewett theorem, for a given $\epsilon > 0$, there is a positive integer $N$such that all subsets of $[3]^N$ of density $\epsilon$ contain a combinatorial line. Map the $[3]^N$ to $\mathbb{R}^2$ using a generic linear map $f$ to obtain a set $X = f([3]^N) \subset \mathbb{R}^2$. By the choice of $f$, collinear points in $X$ correspond to collinear points in $[3]^N$, and $f$ restricted to $[3]^N$ is injective. Therefore, $X$ has no 4 points on a line, and so is in almost general position, but if $S \subset X$ has size at least $\epsilon |X|$, the set $f^{-1}(S) \subset [3]^N$ has density at least $\epsilon$ in $[3]^N$. Therefore, $f^{-1}(S)$ has a line, hence $S = f(f^{-1})(S)$ has 3 collinear points. Since $\epsilon > 0$ was arbitrary, this proves that $\alpha(n,2) = o(n)$.\\
If one tries to generalize this argument to higher dimensions, by mapping $[m]^N$ to $\mathbb{R}^d$, then there will be $m^{d-1}$ cohyperplanar points, and we must have $m^{d-1} = d + 1$ to get almost general position. But the only positive integers that have this property are $(m,d) \in \{(3,2), (2,3)\}$. Taking $m = 2, d = 3$ gives $\alpha(n,3) = o(n)$, as observed by Cardinal, T\'{o}th and Wood~\cite{CardinalTothWood} and otherwise we have too many cohyperplanar points as $m^{d-1} > d + 1$. Overcoming this obstacle is the main goal of the paper.

\subsection{Organization of the paper}
Section 2 is devoted to motivating the arguments of this paper and to explaining the approach in the proofs of the main results. In the section 3, we introduce the key notion of this paper, \emph{$\mathcal{F}-$incident sets}, where $\mathcal{F}$ is an arbitrary family of maps from $\mathbb{R}^N$ to $\mathbb{R}^d$. Roughly, these are the sets that stay cohyperplanar under all maps in $\mathcal{F}$. The basic properties of $\mathcal{F}-$incident sets are studied and we prove the Proposition~\ref{incremfn}, which gives the incidence removal function, a single function which makes all sets non-cohyperplanar, except $\mathcal{F}-$incident sets. In the next section, we specialize to the study of $\mathcal{F}_{N,d,m}-$incident sets, where $\mathcal{F}_{N,d,m}$ is a family of maps from $\mathbb{R}^N$ to $\mathbb{R}^d$ similar to polynomials of $m-$th degree. In particular, in the Lemma~\ref{exInc} we show that combinatorial subspaces and lines in particular give raise to $span \mathcal{F}_{N,d,m}$-incident sets. The rest of the section is devoted to deriving a characterization of $\mathcal{F}_{N,d,m}$-incident sets in terms of vectors given by products of coordinates. The proof of $\alpha(n,d) = o_d(n)$ is the result of work in section 5, which also contains the main tool in the analysis of $\mathcal{F}_{N,d,m}-$incident sets, the Lemma~\ref{algebraicLemma1}. Finally, in the section 6, we improve the bounds for certain dimensions, using the Lemma~\ref{algLemmaSets} in the analysis of $\mathcal{F}_{N,d,m}-$incident sets.

\section{Motivation and the outline of the proof}
Recall that the main obstacle to generalizing F\"{u}redi's argument to higher dimensions is that $d$-cube have too many cohyperplanar points. A possible way to get around this issue is to modify the initial set $[m]^N$ to a subset $X$, which does not have too many incidences, and yet the Hales-Jewett theorem still holds in some form. The desired set would once again be the image of $X$ under a generic linear map from $\mathbb{R}^N \to \mathbb{R}^d$. It is tempting to try to remove certain points from each $(d-1)$-cube, so that precisely $d+1$ out of original $m^d$ remain. However, this is impossible for sufficiently large $N$, as the set $X \subset [m]^N$ gives a 2-colouring of $[m]^N$ (a point is blue if it is in $X$, red otherwise), and thus there is a monochromatic $(d+1)$-cube. Therefore, such an approach at least needs further modifications, if it can be made to work at all.\\

Having abandoned the first idea, it is natural try to map $[d+1]^N$ under a map $f$ which is more general than linear maps. Previously we used a generic linear map, in other words, this is a map which destroys all the cohyperplanarities, except those that are obvious, i.e. cohyperplanar sets in $[d+1]^N$. Our key notion in this paper is \emph{$\mathcal{F}$-incident set}. Let $\mathcal{F}$ be a family of functions from $\mathbb{R}^N$ to $\mathbb{R}^d$ that we are using instead of linear maps only. We say that a set $S \subset \mathbb{R}^N$ is \emph{$\mathcal{F}$-incident} if the multiset $f(S)$ is affinely dependent for all $f \in \mathcal{F}$. Crucially, we have a similar situation with more general maps as that in the case of linear maps, namely we can a find a `generic' map $f \in span \mathcal{F}$, such that if $f(S)$ is affinely dependent then $S$ is $\mathcal{F}-$incident. This is the content of Proposition~\ref{incremfn}, we refer to such a map as the `incidence removal function'.\\

After we have constructed the incidence removal function, the next aim is to study $\mathcal{F}$-incident sets for suitable $\mathcal{F}$. Our goal now is essentially the following: we want that dense subsets of $[m]^N$ contain a $\mathcal{F}$-incident set of size $d+1$ (which gives $d+1$ cohyperplanar sets), but that the image of $[m]^N$ under an incidence removal function does not contain $d+2$ cohyperplanar points. An easy way to fulfill the second requirement is to require that $\mathcal{F}$-incident sets of size $d+1$ cannot have interesection of size $d$. On the other hand, as in the case of linear maps, we use the density Hales-Jewett theorem for the first part, thus we want that the combinatorial subspaces are $span \mathcal{F}-$incident (not only $\mathcal{F}$-incident, as the incidence removal function belongs to $span \mathcal{F}$).\\

To give an idea how we choose the family of functions $\mathcal{F}$ making the combinatorial lines $span \mathcal{F}-$incident, observe the following identities that hold for arbitrary $a,b$:
\begin{description}
\item $\mathbf{1}.1 + (\mathbf{-3}).1 + \mathbf{3}.1 + (\mathbf{-1}).1 = 0$
\item $\mathbf{1}.a + (\mathbf{-3}).(a+b) + \mathbf{3}.(a+2b) + (\mathbf{-1}).(a+3b) = 0$
\item $\mathbf{1}.a^2 + (\mathbf{-3}).(a+b)^2 + \mathbf{3}.(a+2b) + (\mathbf{-1}).(a+3b)^2 = 0.$
\end{description}
What is crucial here is that we have the same coefficients appearing in the three linear combinations above. Hence, if look at $f:\mathbb{R}^N \to \mathbb{R}^3$ of the form 
\begin{dmath}\label{deg2eq}f(x) = \begin{pmatrix}(\langle x, v_1 \rangle + c_1)^2\\(\langle x, v_2 \rangle + c_2)^2\\(\langle x, v_3 \rangle + c_3)^2 \end{pmatrix}\end{dmath}
for some $v_1, v_2, v_3 \in \mathbb{R}^N$ and reals $c_1, c_2, c_3$, then $f(x), f(x+y), f(x+2y), f(x+3y)$ are necessarily cohyperplanar, as
\begin{dmath*}\mathbf{1}.f(x) + (\mathbf{-3}).f(x+y) + \mathbf{3}.f(x+2y) +(\mathbf{-1}).f(x+3y) = 0\end{dmath*}
and the sum of coefficients is zero. Even further, if $g$ is any linear combination of functions of the form described above, then $g(x), g(x+y), g(x+2y), g(x+3y)$ are cohyperplanar, owing to the same coefficients in the above identities.\\

With linear maps and $d=2$, we had that the image of $[3]^N$ to plane under a generic linear map is the desired set, as the combinatorial lines gave colinear sets of points. Now moving to functions that come from polynomials of degree 2, the image of $[4]^N$ under a `generic degree 2 function' to $\mathbb{R}^3$ has cohyperplanar sets of 4 points that are also images of combinatorial lines. After some analysis of $\mathcal{F}$-incident sets for $\mathcal{F}$ given by equation~(\ref{deg2eq}), we are able to show that these have interesection of size at most 1, if the size of sets is at most 4. The motivation for this step comes from the fact that we expect that our non-trivial $\mathcal{F}$-incident sets are precisely the relevant combinatorial subspaces (in this case the lines) and as such, they cannot have large intersection (in case of lines, they cannot share more than one point). This was the second requeirement that we had, sketching the proof that $\alpha(n,3) = o(n)$. This naturally extends to larger values of $d$.\\

Using different identities, we are able to get better bounds on $\alpha(n,d)$. For example, from the fact that $x^2 + (x+a+b)^2 + (x+a+c)^2 + (x+b+c)^2 = (x+a)^2 + (x+b)^2 + (x+c)^2 + (x+a+b+c)^2$, we are able to use 3-dimensional combinatorial subspaces of $\{0,1\}^N$ as the sources of cohyperplanar sets. Generalizing this identity to higher degrees, we can use the higher-dimensional combinatorial subspaces as well. The better bounds in this case come from the better bounds for density Hales-Jewett theorem~\cite{Polymath} in the case of $\{0,1\}^N$, the generalized Sperner's theorem.\\   

When it comes to analysis of $\mathcal{F}$-incident sets, let us first define precisely the families of functions that we shall consider in this paper. For given $N,d,m \in \mathbb{N}$ we define the family $\mathcal{F}_{N,d,m}$ of functions $f:\mathbb{R}^N\to\mathbb{R}^d$ of the form $f_i(x) = (\langle x, u_i \rangle + c_i)^l$ for $i=1, 2, \dots, d$, for any $u_1,u_2, \dots, u_d \in \mathbb{R}^N$, $c_1, c_2, \dots, c_d \in \mathbb{R}$ and $1 \leq l \leq m$.\\

First important claim regarding the $\mathcal{F}_{N,d,m}$-incident sets is the characterization given by the Proposition~\ref{incCritBasic}. To simplify the notation, we introduce the notion \emph{$(\leq m)-$function to $S$} as any function $f:A\to S$, where $A$ has size at most $m$. Given a vector $x \in \mathbb{R}^N$ and a $(\leq m)-$function $f$ to $[N]$, we define $f(x) = \prod_{a \in A} x_{f(a)}$. The Proposition~\ref{incCritBasic} tells us that $\{x_0, x_1, \dots, x_r\}$ for $r \leq d$ is $\mathcal{F}_{N,d,m}$-incident if and only if the vectors 
$$\begin{pmatrix} f_1(x_0)\\ f_2(x_0) \\ \vdots \\f_r(x_0)\end{pmatrix}, \begin{pmatrix} f_1(x_1)\\ f_2(x_1) \\ \vdots \\f_r(x_1)\end{pmatrix}, \dots, \begin{pmatrix} f_1(x_r)\\ f_2(x_r) \\ \vdots \\f_r(x_r)\end{pmatrix}$$
are affinely dependent for all $(\leq m)-$functions $f_1, f_2, \dots, f_r$. Then, in order to prove that our $\mathcal{F}_{N,d,m}-$incident sets cannot have large intersections, we use a Lemma~\ref{algebraicLemma1} and Lemma~\ref{algLemmaSets}. These combinatorial lemmas construct $(\leq m)-$functions which contradict the Proposition~\ref{incCritBasic}.

\section{Definition and basic properties of $\mathcal{F}$-incidences}

Throughout this section, $\mathcal{F}$ will stand for a family of maps from $\mathbb{R}^N$ to $\mathbb{R}^d$. Given such a family of functions $\mathcal{F}$, our goal is to understand the non-trivial affinely dependant sets of points in the images of $f \in \mathcal{F}$.\\
We say that points $s_1, s_2, \dots, s_k \in \mathbb{R}^d$ (not necessarily distinct) are \emph{affinely dependant} if there are $\lambda_1, \dots, \lambda_k \in \mathbb{R}$ not all zero such that $\sum_{i=1}^k \lambda_i = 0$ and $\sum_{i=1}^k \lambda_i s_i = 0$. A $k$-tuple $S = (s_1, s_2, \dots, s_k)$ of points in $\mathbb{R}^N$ is said to be $\mathcal{F}-$\emph{incident} if for all $f \in \mathcal{F}$ we have $f(s_1), f(s_2), \dots, f(s_k)$ affinely dependant. A set $S = \{s_1, s_2, \dots, s_k\}$ of points in $\mathbb{R}^N$ is $\mathcal{F}-$\emph{incident} if a corresponding $k$-tuple $(s_1, s_2, \dots, s_k)$ is. Further, $S$ is \emph{minimal} $\mathcal{F}$-\emph{incident} if it is $\mathcal{F}$-incident and no proper subset of $S$ is $\mathcal{F}-$incident.

\begin{proposition}(Incidence removal function.)\label{incremfn}Let $X \subset \mathbb{R}^N$ be finite and let $\mathcal{F}$ be a family of functions from $\mathbb{R}^N$ to $\mathbb{R}^d$. Then there is $f \in span\mathcal{F}$ with the property that
\begin{description}
\item if $\{s_1, s_2, \dots, s_k\}$ is not $\mathcal{F}$-incident, then $f(s_1), f(s_2), \dots, f(s_k)$ are affinely independent. ($\dagger$)
\end{description}
Furthermore, if $\mathcal{F}$ separates the points of $X$ (i.e. for distinct $x,y \in X$ there is $f \in \mathcal{F}$ such that $f(x) \not = f(y)$), then there is $f \in span \mathcal{F}$ which is injective on $X$, with the property ($\dagger$).\end{proposition}

The proof of the proposition is based on simple linear algebra and some easy facts regarding the vanishing of polynomials. It can be skipped at the first reading, the reader should only be aware of the existence of the incidence removal function and its properties.

\begin{proof} Throughout this proof, for a function $f$ and set $S = \{s_1, s_2, \dots, s_k\}$, we regard $f(S)$ as a multiset of elements $f(s_1), \dots, f(s_k)$. So, if we say that $f(S)$ is affinely dependent, we mean $f(s_1), f(s_2), \dots, f(s_k)$ are affinely dependent.\\
Firstly, we prove the first part of the proposition. The last part will follow from a simple argument later. Let $T_1, T_2, \dots, T_m$ be the list of all subsets of $X$ which are not $\mathcal{F}$-incident. Thus, for each $i$ we have a function $f_i \in \mathcal{F}$ such that $f_i(T_i)$ is affinely independent. We shall inductively construct functions $F_i \in span \mathcal{F}$ such that all of $F_i(T_1), F_i(T_2), \dots, F_i(T_i)$ are affinely independent. Start by taking $F_1 = f_1$ for the case $i=1$.\\
Suppose that we have $i \geq 1$ such that $F_i(T_1), F_i(T_2), \dots, F_i(T_i)$ are affinely independent. Assume that $i < m$, otherwise we are done. Also, if $F_i(T_{i+1})$ is already affinely independent, simply take $F_{i+1} = F_i$. Hence, w.l.o.g. $F_{i}(T_{i+1})$ is affinely dependent. We shall construct $F_{i+1}$ as a linear combination $F_{i} + \lambda f_{i+1}$, where $\lambda > 0$ is sufficiently small so that it does not introduce new dependencies.\\  
Let $u_1, u_2, \dots, u_k \in \mathbb{R}^N$. Let $F^{(\lambda)} = F_{i} + \lambda f_{i+1}$ and suppose that $F^{(0)}(u_1), F^{(0)}(u_2), \dots, F^{(0)}(u_k)$ are affinely independent. Then $F^{(0)}(u_2) - F^{(0)}(u_1), \dots, F^{(0)}(u_k) - F^{(0)}(u_1)$ are linearly independent. 

\begin{lemma}\label{restrLemma}Suppose that $v_1, \dots, v_l \in \mathbb{R}^d$ are linearly independent. Then, we can find $I \subset [d]$ of size $l$ such that $v_1, \dots, v_l$ are still linearly independent when restricted in coordinates in $I$.\end{lemma} 
\begin{proof} Look at the $d \times l$ matrix $A = (v_1 v_2 \dots v_l)$. Since $v_1, v_2, \dots, v_l$ are linearly independent, the column rank of $A$ is $l$. But the column rank is the same as the row rank, so we can find $l$ linearly independent rows $r_1, \dots, r_l$. Take $I = \{r_1, \dots, r_l\}$ and let $A'$ be the matrix $A$ restricted to rows in $I$. Then, the row rank of $A'$ is $l$, so its column rank is $l$, as desired.
\end{proof}

By Lemma~\ref{restrLemma} we can find a set of coordinates $I$ of size $k-1$ such that $F^{(0)}(u_2) - F^{(0)}(u_1), \dots, F^{(0)}(u_k) - F^{(0)}(u_1)$ are linearly independent after restriction. Restrict our attention to these coordinates. Then we can define $p(\lambda) = \det(F^{(\lambda)}(u_2) - F^{(\lambda)}(u_1), \dots, F^{(\lambda)}(u_k) - F^{(\lambda)}(u_1))$, which is a polynomial in $\lambda$. Since $p(0) \not= 0$, by continuity we have $\delta > 0$ such that if $|\lambda| < \delta$ then $p(\lambda) \not= 0$. Therefore, $F^{(\lambda)}(u_1), F^{(\lambda)}(u_2), \dots, F^{(\lambda)}(u_k)$ are affinely independent if $|\lambda| < \delta$.\\
We can apply this argument to all $T_1, \dots, T_i$, to get $\delta > 0$ such that if $|\lambda| < \delta$ then $(F_i + \lambda f_{i+1})(T_j)$ is affinely independent for all $j = 1, \dots, i$.\\
Now suppose that the choice $F_i + \lambda f_{i+1}$ does not work for us as $F_{i+1}$. Then, we must have $(F_i + \lambda f_{i+1}) (T_{i+1})$ is affinely dependent for all $|\lambda| < \delta$. Thus if $\lambda > \delta^{-1}$ then $(\lambda F_i + f_{i+1}) (T_{i+1})$ is affinely dependent. Now, apply the Lemma~\ref{restrLemma} to $f_{i+1}(T_{i+1})$ to get a set of coordinates of size $r-1$, on which this set is still affinely independent, and use a similar polynomial as before, this time $q(\lambda) = \det ((\lambda F_i + f_{i+1}) (t_2 - t_1) \dots (\lambda F_i + f_{i+1}) (t_r - t_1))$, where $T_{i+1} = \{t_1, t_2, \dots, t_r\}$. Then $q(0) \not= 0$, but $q(\lambda) = 0$ if $\lambda > \delta^{-1}$ which is a contradiction, and thus the first part of the proposition is proved.\\

For the last part, if $\mathcal{F}$ separates the points of $X$, observe that there are no two-element sets which are $\mathcal{F}$-incident. Hence, $f(x)$ and $f(y)$ are affinely independent by the first part, so $f$ is injective, as desired. 
\end{proof}

\section{Families of higher-degree maps and the resulting incident sets}
Throughout the rest of the paper we will focus on the family $\mathcal{F}_{N,d,m}$ of functions $f:\mathbb{R}^N\to\mathbb{R}^d$ of the form $f_i(x) = (\langle x, u_i \rangle + c_i)^l$ for $i=1, 2, \dots, d$, for any $u_1,u_2, \dots, u_d \in \mathbb{R}^N$, $c_1, c_2, \dots, c_d \in \mathbb{R}$ and $1 \leq l \leq m$.\\

We start by giving the examples of non-trivial $span \mathcal{F}_{N,d,m}$-incident sets. The proofs are based on algebraic identities, which were described in the introduction. For the case of lines, we use the rank-nullity theorem to prove that there is an identity we are looking for, and in the case of combinatorial subspaces, we prove the identity explicitly.

\begin{lemma}(Examples of non-trivial $span \mathcal{F}_{N,d,m}$-incident sets.)\label{exInc}
\begin{enumerate}
\item (Lines) For $x, y \in \mathbb{R}^N$, the $m+2-$tuple $(x + iy: i = 0, 1, \dots, m+1)$ is $span \mathcal{F}_{N,d,m}-$incident.
\item ($m+1$-dimensional combinatorial subspace) For $x_0, x_1, \dots, x_{m+1} \in \mathbb{R}^N$, the $2^{m+1}$-tuple $(x_0 + \sum_{i \in I} x_i: I \subset [m+1])$ is $span  \mathcal{F}_{N,d,m}-$incident.
\end{enumerate}
\end{lemma}

\begin{proof}\emph{Lines.} We show that there are $\lambda_0, \dots, \lambda_{m+1}$, not all zero, such that for all $f \in \mathcal{F}_{N,d,m}$ we have $\sum_{i=0}^{m+1} \lambda_i f(x + iy) = 0$ and $\sum_{i=0}^{m+1} \lambda_i = 0$. Then, the same linear combination shows that $f(x), f(x+y), \dots, f(x + (m+1)y)$ are affinely dependent for $f \in span \mathcal{F}_{N,d,m}$.\\

Thus, we want non-trivial $\lambda_i$ adding up to zero, such that for all $u \in \mathbb{R}^N, c \in \mathbb{R}, l \in [m]$ we have
\begin{dmath*} \sum_{i=0}^{m+1} \lambda_i \left(\langle x + iy, u\rangle + c\right)^l = 0.\end{dmath*}
This is equivalent to 
\begin{dmath*} \sum_{i=0}^{m+1} \lambda_i \langle x + iy, u\rangle^l = 0\end{dmath*}
for all $u \in \mathbb{R}$ and $l \in [m]$. 
Further, this is equivalent to
\begin{dmath*} \sum_{i=0}^{m+1} \lambda_i i^l = 0\end{dmath*}
for all $l = 0, 1, \dots, m$.
Hence, if $\lambda_0, \dots, \lambda_{m+1}$ satisfy
\begin{dmath*} \sum_{i=0}^{m+1} \lambda_i i^l = 0\end{dmath*}
for all $l = 0, 1, \dots, m$, we are done. But by rank-nullity theorem (`more variables than equations'), we must have non-trivial solution to these equations, giving us the desired $\lambda_i$.\\

\emph{$m+1$-dimensional subspaces.} As in the case of lines, we show that there are $\lambda_I, I\subset [m+1]$, not all zero, but adding up to zero, such that $\sum_{I \subset[m+1]} \lambda_I f(x_0 + \sum_{i \in I} x_i) = 0$, for all $f \in \mathcal{F}_{N,d,m}$, which suffices to prove the claim in the full generality. In this case, we can actually set $\lambda_I = (-1)^{|I|}$.\\
It is enough to show that for any $u \in \mathbb{R}^N, c \in \mathbb{R}, l \in [m]$ we have 
\begin{dmath*}\sum_{I \subset [m+1]} (-1)^{|I|} (\langle x_0 + \sum_{i \in I} x_i, u\rangle + c)^l = 0.\end{dmath*}
But writing $a_0 = \langle x_0, u\rangle + c, a_i = \langle x_i, u \rangle$ for $i = 1, \dots, m+1$, we see that it is sufficient to show 
\begin{dmath*}\sum_{I \subset [m+1]} (-1)^{|I|} \left(a_0 + \sum_{i\in I} a_i\right)^l = 0\end{dmath*}
for all $a_0, a_1, \dots, a_{m+1} \in \mathbb{R}, l \in [m+1]$. This is the content of the next lemma.\\
\begin{lemma}\label{lemmaIdCube}Let $l, m \in \mathbb{N}, l \leq m$ and $a_0, a_1, \dots, a_{m+1} \in\mathbb{R}$. Then 
\begin{dmath*}\sum_{I \subset [m+1]} (-1)^{|I|} \left(a_0 + \sum_{i\in I} a_i\right)^l = 0.\end{dmath*}
\end{lemma}
\begin{proof}[Proof of Lemma~\ref{lemmaIdCube}.]
Note that 
\begin{dmath*}\sum_{I \subset [m+1]} (-1)^{|I|} \left(a_0 + \sum_{i\in I} a_i\right)^l
= \sum_{k = 0}^l \binom{l}{k}a_0^k \sum_{I \in [m+1]} (-1)^{|I|} \left(\sum_{i\in I} a_i\right)^k\end{dmath*}
thus we only need to consider the case $a_0 = 0$.\\
Consider the expression
\begin{dmath*}\sum_{I \in [m+1]} (-1)^{|I|} \left(\sum_{i\in I} a_i\right)^l\end{dmath*}
as a polynomial of degree $l$ in $a_1, \dots, a_{m+1}$. The coefficient of $a_1^{d_1} a_2^{d_2}\dots a_{m+1}^{d_{m+1}}$ is \begin{dmath*}\binom{l}{d_1,d_2, \dots, d_{m+1}} \sum_{S \subset I \subset [m+1]} (-1)^{|I|},\end{dmath*}
where $S$ is the set of indices $i$ such that $d_i > 0$. Since $|S| \leq m$, the sum $\sum_{S \subset I \subset [m+1]} (-1)^{|I|}$ is zero, which finishes the proof.\end{proof}
Applying the Lemma~\ref{lemmaIdCube} completes the proof.\end{proof}

Before coming to a key proposition which describes the $\mathcal{F}_{N,d,m}-$incident sets, we introduce a couple of pieces of notation. If $f$ is a function from a set of size at most $m$ to a set $X$, we say that $f$ is a \emph{$(\leq m)$-function} to $X$. Given a $(\leq m)$-function $f:A\to[N]$ and $x \in \mathbb{R}^N$ we write $f(x) = \prod_{a \in A} x_{f(a)}$. Here we allow an `empty' function, i.e. a function $f$ from an empty set to $[N]$, defining $f(x) = 1$, for all $x \in \mathbb{R}^N$.

\begin{proposition}\label{incCritBasic}Let $r,d,m,N \in \mathbb{N}$, suppose $r \leq d$ and let $X = \{x_0, x_1, \dots, x_r\} \subset \mathbb{R}^N$. The following are equivalent.
\begin{description}
\item[(i)] $X$ is $\mathcal{F}_{N,d,m}$-incident.
\item[(ii)] Given any $(\leq m)-$functions $f_1, f_2, \dots, f_r$ to $[N]$, the vectors
$$\begin{pmatrix} f_1(x_0)\\ f_2(x_0) \\ \vdots \\f_r(x_0)\end{pmatrix}, \begin{pmatrix} f_1(x_1)\\ f_2(x_1) \\ \vdots \\f_r(x_1)\end{pmatrix}, \dots, \begin{pmatrix} f_1(x_r)\\ f_2(x_r) \\ \vdots \\f_r(x_r)\end{pmatrix}$$
are affinely dependent.
\end{description}
\end{proposition}

The proof of the proposition is a straightforward algebraic manipulation, mostly based on the fact that if a polynomial over the reals vanishes everywhere, its coefficients are zero. The reader may consider skipping the proof in the first reading. 

\begin{proof}Start from the definition, \textbf{(i)} is equivalent to the vectors
$$\begin{pmatrix} (\langle x_0, u_1 \rangle + c_1)^l\\ (\langle x_0, u_2 \rangle + c_2)^l\\ \vdots \\ (\langle x_0, u_d \rangle + c_d)^l  \end{pmatrix}, \begin{pmatrix} (\langle x_1, u_1 \rangle + c_1)^l\\ (\langle x_1, u_2 \rangle + c_2)^l\\ \vdots \\ (\langle x_1, u_d \rangle + c_d)^l  \end{pmatrix}, \dots, \begin{pmatrix} (\langle x_r, u_1 \rangle + c_1)^l\\ (\langle x_r, u_2 \rangle + c_2)^l\\ \vdots \\ (\langle x_r, u_d \rangle + c_d)^l  \end{pmatrix}$$
being affinely dependent for any choice of parameters $c_1, c_2, \dots, c_d \in \mathbb{R}$, $u_1, u_2, \dots, u_d \in \mathbb{R}^N$ and $1 \leq l \leq m$. In particular, as $r \leq d$, this is further equivalent to vectors
$$\begin{pmatrix} (\langle x_1, u_1 \rangle + c_1)^l - (\langle x_0, u_1 \rangle + c_1)^l\\ (\langle x_1, u_2 \rangle + c_2)^l - (\langle x_0, u_2 \rangle + c_2)^l\\ \vdots \\ (\langle x_r, u_r \rangle + c_r)^l-(\langle x_0, u_r \rangle + c_r)^l  \end{pmatrix}, \begin{pmatrix} (\langle x_2, u_1 \rangle + c_1)^l - (\langle x_0, u_1 \rangle + c_1)^l\\ (\langle x_2, u_2 \rangle + c_2)^l - (\langle x_0, u_2 \rangle + c_2)^l\\ \vdots \\ (\langle x_r, u_r \rangle + c_r)^l - (\langle x_0, u_r \rangle + c_r)^l  \end{pmatrix}, \dots, \begin{pmatrix} (\langle x_r, u_1 \rangle + c_1)^l - (\langle x_0, u_1 \rangle + c_1)^l\\ (\langle x_r, u_2 \rangle + c_2)^l - (\langle x_0, u_2 \rangle + c_2)^l\\ \vdots \\ (\langle x_r, u_r \rangle + c_r)^l-(\langle x_0, u_r \rangle + c_r)^l \end{pmatrix}$$
being linearly dependent for all choices of parameters. Hence, taking determinant, \textbf{(i)} is the same as
$$\det \left((\langle x_i, u_j \rangle + c_j)^l - (\langle x_0, u_j \rangle + c_j)^l\right) = 0$$
for any choice of $u_1, \dots, u_r, c_1, \dots, c_r, l$. Expanding, we obtain
\begin{dmath*}0 = \sum_{\pi \in S_r} sgn(\pi) \prod_{i=1}^r \left((\langle x_{\pi(i)}, u_i \rangle + c_i)^l - (\langle x_0, u_i \rangle + c_i)^l\right)
= \sum_{\pi \in S_r} sgn(\pi) \prod_{i=1}^r \left(\sum_{k = 0}^l c_i^k \binom{l}{k} \left(\langle x_{\pi(i)}, u_i \rangle^{l-k} - \langle x_0, u_i \rangle^{l-k}\right)\right)
=  \sum_{0 \leq k_1, k_2, \dots, k_r \leq l} c_1^{k_1} c_2^{k_2} \dots c_r^{k_r} \prod_{i=1}^r \binom{l}{k_i} \left(\sum_{\pi \in S_r} sgn(\pi) \prod_{i=1}^r \left(\langle x_{\pi(i)}, u_i \rangle^{l-k_i} - \langle x_0, u_i \rangle^{l-k_i}\right) \right) 
\end{dmath*}
However, this holds for any choice of $c_1, c_2, \dots, c_r \in \mathbb{R}$, so, when the expression above is viewed as a polynomial in variables $c_1, c_2, \dots, c_r$, we conclude that the coefficients are zero. In other words, \textbf{(i)} is equivalent to the following. For any $0 \leq k_1, k_2, \dots, k_r \leq m$, and any $u_1, u_2, \dots, u_r \in \mathbb{R}^N$ we have
\begin{dmath*}0 = \sum_{\pi \in S_r} sgn(\pi) \prod_{i=1}^r \left(\langle x_{\pi(i)}, u_i \rangle^{k_i} - \langle x_0, u_i \rangle^{k_i}\right)
= \sum_{\pi \in S_r} sgn(\pi) \prod_{i=1}^r \left((\sum_{j = 1}^N x_{\pi(i) j} u_{ij})^{k_i} -(\sum_{j = 1}^N x_{0 j} u_{ij})^{k_i}\right)
= \sum_{\pi \in S_r} sgn(\pi) \prod_{i=1}^r \left(\sum_{f:[k_i] \to [N]} \left(\prod_{j=1}^{k_i} x_{\pi(i) f(j)} u_{i f(j)} - \prod_{j=1}^{k_i} x_{0 f(j)} u_{i f(j)}\right) \right)
= \sum_{\pi \in S_r} sgn(\pi) \prod_{i=1}^r \left(\sum_{f:[k_i] \to [N]} \left(\prod_{j=1}^{k_i} u_{if(j)}\right)\left(\prod_{j=1}^{k_i} x_{\pi(i) f(j)} - \prod_{j=1}^{k_i} x_{0 f(j)}\right)\right)
= \sum_{f_1:[k_1]\to[N], \dots, f_r:[k_r] \to [N]} \left(\prod_{i=1}^r \prod_{j=1}^{k_i} u_{i f_i(j)}\right)\left(\sum_{\pi \in S_r} sgn(\pi) \prod_{i=1}^r\left(\prod_{j=1}^{k_i} x_{\pi(i) f_i(j)} - \prod_{j=1}^{k_i} x_{0 f_i(j)}\right)\right) 
= \sum_{f_1:[k_1]\to[N], \dots, f_r:[k_r] \to [N]} \left(\prod_{i=1}^r f_i(u_i) \right) \left(\sum_{\pi \in S_r} sgn(\pi) \prod_{i=1}^r (f_i(x_{\pi(i)}) - f_i(x_0))\right)
\end{dmath*}    
Now, look at the expression above as a polynomial in variables $u_{ij}$. Observe that if $f_1, f_2, \dots, f_r, g_1, g_2, \dots, g_r$ are such that $\prod_{i=1}^r f_i(u_i) = \prod_{i=1}^r g_i(u_i)$ as formal expressions, then we must have $\sum_{\pi \in S_r} sgn(\pi) \prod_{i=1}^r (f_i(x_{\pi(i)}) - f_i(x_0)) = \sum_{\pi \in S_r} sgn(\pi) \prod_{i=1}^r (g_i(x_{\pi(i)}) - g_i(x_0))$ as well. This tells us that the coefficients of our polynomial are positive integer multiples of $\sum_{\pi \in S_r} sgn(\pi) \prod_{i=1}^r (f_i(x_{\pi(i)}) - f_i(x_0))$. Also, the polynomial over $\mathbb{R}$ vanishes everywhere iff its coefficients are zero. Therefore, \textbf{(i)} holds iff for all $(\leq m)$-functions $f_1, f_2, \dots, f_r$ to $[N]$, we have
\begin{dmath*}0 = \sum_{\pi \in S_r} sgn(\pi) \prod_{i=1}^r (f_i(x_{\pi(i)}) - f_i(x_0))
= \det_{1\leq i,j \leq r}(f_i(x_j) - f_i(x_0))
\end{dmath*}
which says precisely that the vectors
$$\begin{pmatrix} f_1(x_1) - f_1(x_0) \\ f_2(x_1) - f_2(x_0) \\ \vdots \\ f_r(x_1) - f_r(x_0) \end{pmatrix}, \begin{pmatrix} f_1(x_2) - f_1(x_0) \\ f_2(x_2) - f_2(x_0) \\ \vdots \\ f_r(x_2) - f_r(x_0) \end{pmatrix}, \dots, \begin{pmatrix} f_1(x_r) - f_1(x_0) \\ f_2(x_r) - f_2(x_0) \\ \vdots \\ f_r(x_r) - f_r(x_0) \end{pmatrix}$$
are linearly dependent, which is equivalent to \textbf{(ii)}, as desired.\end{proof}

\begin{proposition}\label{incCriterion}Let $r,d,m,N \in \mathbb{N}$ and suppose $r \leq d$. Suppose that $\{x_0, x_1, \dots, x_r\} \subset \mathbb{R}^N$ is $\mathcal{F}_{N,d,m}$-incident. Then, given any affine map $\alpha:\mathbb{R}^N\to \mathbb{R}^N$ and any $(\leq m)-$functions $f_1, f_2, \dots, f_r$ to $[N]$, the vectors
$$\begin{pmatrix} f_1(\alpha(x_0))\\ f_2(\alpha(x_0)) \\ \vdots \\f_r(\alpha(x_0))\end{pmatrix}, \begin{pmatrix} f_1(\alpha(x_1))\\ f_2(\alpha(x_1)) \\ \vdots \\f_r(\alpha(x_1))\end{pmatrix}, \dots, \begin{pmatrix} f_1(\alpha(x_r))\\ f_2(\alpha(x_r)) \\ \vdots \\f_r(\alpha(x_r))\end{pmatrix}$$
are affinely dependent.\\
On the other hand, if $\{x_0, x_1, \dots, x_r\} \subset \mathbb{R}^N$ is not $\mathcal{F}_{N,d,m}$-incident, then, given any affine isomorphism $\alpha:\mathbb{R}^N\to \mathbb{R}^N$, we may find $(\leq m)-$functions $f_1, f_2, \dots, f_r$ to $[N]$, so that the vectors
$$\begin{pmatrix} f_1(\alpha(x_0))\\ f_2(\alpha(x_0)) \\ \vdots \\f_r(\alpha(x_0))\end{pmatrix}, \begin{pmatrix} f_1(\alpha(x_1))\\ f_2(\alpha(x_1)) \\ \vdots \\f_r(\alpha(x_1))\end{pmatrix}, \dots, \begin{pmatrix} f_1(\alpha(x_r))\\ f_2(\alpha(x_r)) \\ \vdots \\f_r(\alpha(x_r))\end{pmatrix}$$
are affinely independent.\\
\end{proposition}

\begin{proof}
Given arbitrary affine map $\alpha:\mathbb{R}^N \to \mathbb{R}^N$, written in the form $\alpha = A + v$ for $N\times N$ matrix $A$ and a vector $v \in \mathbb{R}^N$, vectors $u_1, u_2, \dots, u_r \in \mathbb{R}^N$, constants $c_1, c_2, \dots, c_r \in \mathbb{R}$ and $1 \leq l \leq m$, we have 
$$(\langle \alpha(x), u_i \rangle + c_i)^l = (\langle Ax + v, u_i\rangle + c_i)^l = (\langle x, A^Tu \rangle + (\langle v, u_i \rangle + c_i))^l.$$
But then, since $x_0, x_1, \dots, x_r$ is $\mathcal{F}_{N,d,m}-$incident, it follows that so is $\alpha(x_0), \alpha(x_1), \dots, \alpha(x_r)$. Apply the the Proposition~\ref{incCritBasic} to $\alpha(x_0), \alpha(x_1), \dots, \alpha(x_r)$, from which the first claim in the proposition follows.\\

For the second part, observe that if $\alpha(x_0), \alpha(x_1), \dots, \alpha(x_r)$ is $\mathcal{F}_{N,d,m}-$incident, then by the previous arguments, so is $x_0 = \alpha^{-1}(\alpha(x_0)), \alpha^{-1}(\alpha(x_1)), \dots, \alpha^{-1}(\alpha(x_r))$. Therefore, $\alpha(x_0), \alpha(x_1), \dots, \alpha(x_r)$ is not $\mathcal{F}_{N,d,m}-$incident. The Proposition~\ref{incCritBasic} applies, and gives the desired $(\leq m)-$functions.\end{proof}

\section{Proof of $\alpha(n,d) = o_d(n)$}
\begin{lemma}\label{algebraicLemma1}Let $m, r, N \in \mathbb{N}$. Suppose that $y_1, y_2, \dots, y_r \in \mathbb{R}^N$ are vectors such that $rank\{y_1, y_2, \dots, y_r\} + m - 1 \geq r$. Suppose further that $y_1, y_2, \dots, y_r$ are distinct and have non-zero coordinates. Then we may find $(\leq m)-$functions $f_1, f_2, \dots, f_r$ for which the vectors
\begin{dmath*}\begin{pmatrix} f_1(y_1)\\ f_2(y_1) \\ \vdots \\f_r(y_1)\end{pmatrix}, \begin{pmatrix} f_1(y_2)\\ f_2(y_2) \\ \vdots \\f_r(y_2)\end{pmatrix}, \dots, \begin{pmatrix} f_1(y_r)\\ f_2(y_r) \\ \vdots \\f_r(y_r)\end{pmatrix}\end{dmath*}
are linearly independent.\end{lemma}

\begin{proof} We prove the lemma by induction, first on $m$, then on $r$. The lemma holds for $m=1$, this just says that for $r$ linearly independent vectors, we may pick $r$ coordinates, so that after restriction the vectors are still linearly independent -- this is precisely the Lemma~\ref{restrLemma}. Suppose now that the claim holds for some $m-1 \geq 1$. For fixed $m$, we prove the lemma by induction on $r \geq 1$. If $r = 1$, then, take $f:[1] \to [N]$, given by $f(1) = 1$, so the vector $(f(y_1))$ is non-zero.\\

Suppose that the claim holds for some $r \geq 1$, and that $\{y_1, y_2, \dots, y_{r+1}\}$ satisfy the conditions of the lemma.\\

\emph{Case 1.} $y_{r+1} \notin span \{y_1, y_2, \dots, y_r\}$. Then $r + 1 \leq rank\{y_1, y_2, \dots, y_{r+1}\} + m - 1 = rank\{y_1, y_2, \dots, y_r\} + m$, hence $rank\{y_1, y_2, \dots, y_r\} + m - 1 \geq r$. By induction hypothesis, we have $(\leq m)$-functions $f_1, f_2, \dots, f_r$ such that 
\begin{dmath*}\begin{pmatrix} f_1(y_1)\\ f_2(y_1) \\ \vdots \\f_r(y_1)\end{pmatrix}, \begin{pmatrix} f_1(y_2)\\ f_2(y_2) \\ \vdots \\f_r(y_2)\end{pmatrix}, \dots, \begin{pmatrix} f_1(y_r)\\ f_2(y_r) \\ \vdots \\f_r(y_r)\end{pmatrix}\end{dmath*}
are linearly independent. Hence, there are unique $\lambda_1, \lambda_2, \dots, \lambda_r \in \mathbb{R}$ such that $f_i(y_{r+1}) = \sum_{j= 1}^r \lambda_j f_i(y_j)$ holds for all $i=1,\dots,r$. But, $y_{r+1}$ is not in the span of $\{y_1, \dots, y_r\}$, and so $y_{r+1} \not= \sum_{j=1}^r \lambda_j y_j$. Hence, we can pick $f_{r+1}:[1]\to[N]$ to be $f(1) = c$, where $c$ is the coordinate such that
$y_{r+1 c} \not= \sum_{j=1}^r \lambda_j y_{j c}$, finishing the proof in this case.\\

\emph{Case 2.} $y_{r+1} \in span \{y_1, y_2, \dots, y_r\}$. Then $r + 1 \leq rank\{y_1, y_2, \dots, y_{r+1}\} + m - 1 = rank\{y_1, y_2, \dots, y_r\} + m-1$, so 
\begin{dmath*}r \leq rank\{y_1, y_2, \dots, y_r\} + m - 2.\end{dmath*}
By induction hypothesis, we have $(\leq m-1)-$functions $f_1, \dots, f_r$ for which 
\begin{dmath*}\begin{pmatrix} f_1(y_1)\\ f_2(y_1) \\ \vdots \\f_r(y_1)\end{pmatrix}, \begin{pmatrix} f_1(y_2)\\ f_2(y_2) \\ \vdots \\f_r(y_2)\end{pmatrix}, \dots, \begin{pmatrix} f_1(y_r)\\ f_2(y_r) \\ \vdots \\f_r(y_r)\end{pmatrix}\end{dmath*}
are linearly independent. As before, there are unique $\lambda_1, \lambda_2, \dots, \lambda_r \in \mathbb{R}$ such that $f_i(y_{r+1}) = \sum_{j= 1}^r \lambda_j f_i(y_j)$ holds for all $i=1,\dots,r$.\\
We try to take $f_{r+1}$ to be some $f_i$ with additional element in the domain, mapped to $c \in [N]$. If this works, we are done. Otherwise, for all $i=1, \dots, r$ and $c \in [N]$, we have $f_i(y_{r+1}) y_{r+1 c} = \sum_{j = 1}^r \lambda_j f_i(y_j) y_{j c}$. Since the coordinates are non-zero, we get 
\begin{dmath*}f_i(y_{r+1}) = \sum_{j = 1}^r \left(\lambda_j y_{j c}/y_{r+1 c} \right)f_i(y_j).\end{dmath*}
But, by uniqueness of $\lambda_j$, we must have $\lambda_j y_{j c}/y_{r+1 c} = \lambda_j$ for all $j, c$. If some $\lambda_j \not= 0$, then for all $c$ we get $y_{j c}/y_{r+1 c} = 1$, i.e. $y_{r+1} = y_j$ which is a contradiction, as our vectors are distinct. Otherwise, all the $\lambda_j = 0$, so $f_1(y_{r+1}) = 0$, but coordinates of $y_{r+1}$ are non-zero, resulting in contradiction once again.\end{proof}

As a corollary of the algebraic lemma above, we have a result that is consistent with the intuition described in the introduction: we expect lines in $[m+1]^N$ to be the sources of non-trivial $\mathcal{F}_{N, m+1, m}-$incident sets. In other words, a $\mathcal{F}_{N, m+1, m}-$incident set is either larger than $m+1$, and thus its image must be affinely dependent (by looking at dimension of the target space), or the set is on a line.

\begin{corollary}\label{minincline}Suppose that $S \subset \mathbb{R}^N$ is $\mathcal{F}_{N, m+1, m}-$incident. Then, $|S| \geq m+2$ and if $|S| = m+2$, then $S$ is a subset of a line.\end{corollary}

\begin{proof}If $|S| \geq m+3$, we are done. Suppose now that $|S|\leq m+2$. Let $s_0, s_1, \dots, s_{m+1}$ be the elements of $S$. We can find an affine isomorphism $\alpha:\mathbb{R}^N \to \mathbb{R}^N$ such that $\alpha(s_0) = 0$, and $y_i = \alpha(s_i)$, for $i=1,2,\dots, m+1$, are distinct and have non-zero coordinates. By the Proposition~\ref{incCriterion} (note that we may apply it because $|S|-1 \leq m+1$, and $m+1$ is the dimension of the target space), the vectors
\begin{dmath*}\begin{pmatrix} f_1(y_1)\\ f_2(y_1) \\ \vdots \\f_{m+1}(y_1)\end{pmatrix}, \begin{pmatrix} f_1(y_2)\\ f_2(y_2) \\ \vdots \\f_{m+1}(y_2)\end{pmatrix}, \dots, \begin{pmatrix} f_1(y_{m+1})\\ f_2(y_{m+1}) \\ \vdots \\f_{m+1}(y_{m+1})\end{pmatrix}\end{dmath*}
are linearly dependent, for any choice of $(\leq m)-$functions $f_1, f_2,\dots, f_{m+1}$ to $[N]$. Thus, we obtain a contradiction by the Lemma~\ref{algebraicLemma1}, unless
\begin{dmath*}rank\{y_1, y_2, \dots, y_{m+1}\} + m - 1 \leq m.\end{dmath*}
So $rank\{y_1, y_2, \dots, y_{m+1}\} \leq 1$, and as $y_1 \not = 0$, there are scalars $\lambda_1, \dots, \lambda_{m+1}$ such that $y_i = \lambda_i y_1$ holds for all $i=1,\dots, m+1$. But, since $\alpha$ is an affine isomorphism, the points $s_0 = \alpha^{-1}(0), s_1 = \alpha^{-1}(y_1), \dots, s_{m+1} = \alpha^{-1}(y_{m+1})$ are on a line, as desired.
\end{proof}

\begin{theorem}For $d,n \in \mathbb{N}$, $d \geq 2$, we have $\alpha(n,d) = o_d(n)$.\end{theorem}
\begin{proof}Let $\epsilon> 0$ and let $N$ be sufficiently large so that $\epsilon-$density Hales-Jewett theorem holds for combinatorial lines in $[m+2]^N$. Let $X = [m+2]^N$, and let $f$ be a function given by the Proposition~\ref{incremfn} applied to $X$ and $\mathcal{F}_{N,m+1,m}$. Since $\mathcal{F}_{N,m+1,m}$ separates the points of $X$, we may assume that $f$ is injective on $X$. Finally, let $Y = f(X) \subset \mathbb{R}^{m+1}$. We claim that $Y$ has no more than $m+2$ points in a hyperplane, and that all subsets of $Y$ of size at least $\epsilon |Y|$ have a hyperplane with $m+2$ points inside.\\

\emph{There are no more than $m+2$ points of $Y$ on a hyperplane.} Look at a hyperplane $H$ and suppose that $Y$ has $m+3$ points $y_1, \dots, y_{m+3}$ inside $H$. Look at maximal affinely independent subset of $y_1, \dots, y_{m+3}$, w.l.o.g. this is $y_1, y_2, \dots, y_r$ for some $r$. Since $H$ is $m-$dimensional affine subspace, we have $r \leq m+1$. So $S_1 = \{y_1, y_2, \dots, y_r, y_{m+2}\}$ is affinely dependent, and has size at most $m+2$. Then, by definition of $f$ and Proposition~\ref{incremfn}, $T_1 = f^{-1}(S_1)$ is $\mathcal{F}_{N, m+1, m}-$incident. Since $f$ is a bijection from $X$ onto its image, $T_1$ has size at most $m+2$, so by Corollary~\ref{minincline}, $T_1$ is a subset of a line, and $|T_1| = m+2$ and $r = m+1$. Applying the same arguments to $S_2 = \{y_1, \dots, y_r, y_{m+3}\}$ and $T_2 = f^{-1}(S_2)$, we have that $T_2$ is also a subset of a line and has size $m+2$ and also $|T_1 \cap T_2| = m+1$. But, as $T_1, T_2 \subset [m+2]^N$, this is impossible and we have a contradiction, so $Y$ has no more than $m+2$ points on a hyperplane.\\

\emph{Dense subsets of $Y$ are not in general position.} Let $S \subset Y$ have size at least $\epsilon |Y|$. Then $T = f^{-1}(S)$ has a combinatorial line $L$ by the density Hales-Jewett theorem. Hence, $f(L) \subset S$ and $f(S)$ has $m+2$ points that lie on the same hyperplane, by the Lemma~\ref{exInc}. This finishes the proof.\end{proof}

\section{Better bounds for certain dimensions}
In this section, we provide better bounds on $\alpha(n,d)$ for certain dimensions $d$. The key difference in this approach is use of a more efficient version of density Hales-Jewett theorem. 

\begin{theorem}[Generalized Sperner's Theorem,~\cite{Polymath}, Theorem 2.3]\label{sperner} Let $\mathcal{A}$ be a collection of subsets of $[n]$ that contains no $d$-dimensional combinatorial subspace. Then the size of $\mathcal{A}$ is at most $(25/n)^{1/2^d} 2^n$.\end{theorem}  

Here, we consider the points in $\{0,1\}^N \subset \mathbb{R}^N$, which we also interpret as subsets of $[N]$. Observe that, given an $(\leq m)-$function $f$ to $[N]$, with image $S \subset [N]$ and a point $x \in \{0,1\}^N$ corresponding to $X \subset [N]$, we have
$$f(x) = 1_{S\subset X}.$$
Hence, we can reinterpret the Proposition~\ref{incCriterion} in the language of sets as follows. Suppose that $\emptyset, X_1, X_2, \dots, X_r$ correspond to $r+1$ points that are not $\mathcal{F}_{N,d,m}-$incident (so the first point is 0). Then, there are sets $S_1, S_2, \dots, S_r \subset [N]$ of size at most $m$, for which the vectors
\begin{dmath*}\begin{pmatrix}\mathbbm{1}_{S_1 \subset \emptyset}\\\mathbbm{1}_{S_2 \subset \emptyset} \\\vdots\\\mathbbm{1}_{S_r \subset \emptyset}  \end{pmatrix},
\begin{pmatrix}\mathbbm{1}_{S_1 \subset X_1}\\\mathbbm{1}_{S_2 \subset X_1} \\\vdots\\\mathbbm{1}_{S_r \subset X_1}  \end{pmatrix}, \begin{pmatrix}\mathbbm{1}_{S_1 \subset X_2}\\\mathbbm{1}_{S_2 \subset X_2} \\\vdots\\\mathbbm{1}_{S_r \subset X_2}  \end{pmatrix}, \dots, \begin{pmatrix}\mathbbm{1}_{S_1 \subset X_r}\\\mathbbm{1}_{S_2 \subset X_r} \\\vdots\\\mathbbm{1}_{S_r \subset X_r} \end{pmatrix}\end{dmath*}
are affinely independent. If all the sets $S_i$ are non-empty, then the vectors 
\begin{dmath*}\begin{pmatrix}\mathbbm{1}_{S_1 \subset X_1}\\\mathbbm{1}_{S_2 \subset X_1} \\\vdots\\\mathbbm{1}_{S_r \subset X_1}  \end{pmatrix}, \begin{pmatrix}\mathbbm{1}_{S_1 \subset X_2}\\\mathbbm{1}_{S_2 \subset X_2} \\\vdots\\\mathbbm{1}_{S_r \subset X_2}  \end{pmatrix}, \dots, \begin{pmatrix}\mathbbm{1}_{S_1 \subset X_r}\\\mathbbm{1}_{S_2 \subset X_r} \\\vdots\\\mathbbm{1}_{S_r \subset X_r} \end{pmatrix}\end{dmath*}
are linearly independent. Otherwise, w.l.o.g. $S_1 = S_2 = \dots = S_k = \emptyset$ and others are non-empty, so after subtracting the first vector from the others we obtain that 
\begin{dmath*}\begin{pmatrix}0\\0\\\vdots\\0\\\mathbbm{1}_{S_{k+1} \subset X_1} \\\vdots\\\mathbbm{1}_{S_r \subset X_1}  \end{pmatrix},
\begin{pmatrix}0\\0\\\vdots\\0\\\mathbbm{1}_{S_{k+1} \subset X_2} \\\vdots\\\mathbbm{1}_{S_r \subset X_2}  \end{pmatrix},
\dots, \begin{pmatrix}0\\0\\\vdots\\0\\\mathbbm{1}_{S_{k+1} \subset X_r} \\\vdots\\\mathbbm{1}_{S_r \subset X_r}  \end{pmatrix}\end{dmath*}
are linearly independent, which is not possible (when viewed as a matrix, the row rank is less than $r$). This leads us to the following observation.

\begin{observation}\label{setobs}Suppose that the sets $\emptyset, X_1, X_2, \dots, X_r \subset [N]$ correspond to $r+1$ points that are not $\mathcal{F}_{N,d,m}-$incident. Then, there are non-empty sets $S_1, S_2, \dots, S_r \subset \mathbb{N}$ of size at most $m$ such that the vectors
\begin{dmath*}\begin{pmatrix}\mathbbm{1}_{S_1 \subset X_1}\\\mathbbm{1}_{S_2 \subset X_1} \\\vdots\\\mathbbm{1}_{S_r \subset X_1}  \end{pmatrix}, \begin{pmatrix}\mathbbm{1}_{S_1 \subset X_2}\\\mathbbm{1}_{S_2 \subset X_2} \\\vdots\\\mathbbm{1}_{S_r \subset X_2}  \end{pmatrix}, \dots, \begin{pmatrix}\mathbbm{1}_{S_1 \subset X_r}\\\mathbbm{1}_{S_2 \subset X_r} \\\vdots\\\mathbbm{1}_{S_r \subset X_r} \end{pmatrix}\end{dmath*}
are linearly independent.\end{observation}

Viewing these vectors together as an $r\times r$ matrix, we have found that the nullity of this matrix is related to the notion of $\mathcal{F}_{N,d,m}-$incidence. This motivates the study of nullity of such matrices. Before stating the lemma which contains some basic results regarding this problem, we introduce some notation.\\
Given sets $A_1, A_2, \dots, A_r, B_1, B_2, \dots, B_s \in \mathbb{N}^{(<\omega)}$, we write
$$I(A_1, A_2, \dots, A_r; B_1, B_2, \dots, B_s)$$
for the matrix $I_{ij} = \mathbbm{1}_{B_i \subset A_j}$. Further, we define 
$$K(A_1, A_2, \dots, A_r; B_1, B_2, \dots, B_s)$$
as the kernel of $I$ and 
$$n(A_1, A_2, \dots, A_r; B_1, B_2, \dots, B_s)$$ as the nullity of $I$. Also, if $A, B$ are finite sequences of finite sets, of leghts $r$ and $s$, we write $I(A, B) =I(A_1, A_2, \dots, A_r; B_1, B_2, \dots, B_s)$, and similarly we define $K(A,B), n(A,B)$.  

\begin{lemma}\label{algLemmaSets}Let $m, k \in \mathbb{N}$. Given any distinct sets $X_1, X_2, \dots,X_r \in \mathbb{N}^{(<\omega)}$, we can find sets $S_1, S_2, \dots, S_r \subset \mathbb{N}^{(\leq m)}$ which enjoy the following property.
\begin{description}
\item[(i)] $n(X_1, X_2, \dots, X_r; S_1, S_2, \dots, S_r) = 0$, provided $r < 2^{m+1}$.
\item[(ii)] $n(X_1, X_2, \dots, X_r; S_1, S_2, \dots, S_r) \leq 1$, provided $r < 3.2^{m}$. 
\end{description}
\end{lemma}

We prove the lemma by induction and compressions, and in fact use the part (\textbf{i}) in order to deduce the part (\textbf{ii}). As it will be stressed in the proof, there is a subtlety in proving $n(X_1, X_2, \dots, X_r; S_1, S_2, \dots, S_r) \leq 1$, since the naive application of induction only gives $n(X_1, X_2, \dots, X_r; S_1, S_2, \dots, S_r) \leq 2$. The first part provides the required saving of 1 on the RHS.

\begin{proof}\emph{Part (i).} We prove the claim by induction on $\sum_{i=1}^r |X_i|$. If this is zero, then we have $r=1$ and $X_1 = \emptyset$, so just take $S_1 = \emptyset$.\\

Suppose that the lemma holds for smaller values of $\sum_{i=1}^r |X_i|$. Let $x \in \mathbb{N}$ be any element that is contained in at least one of the sets $X_i$. Denote by $\{Y_1, Y_2, \dots, Y_u\}$ the collection of sets given by $\{X_i \setminus \{x\}: i=1,\dots, r\}$, and further let $\{Z_1, \dots, Z_v\}$ be the set $\{X_i: x \notin X_i, X_i \cup \{x\} = X_j\text{ for some }j\}$. Thus $v \leq u$ and $u+v = r$. By induction hypothesis, there are relevant sets $S_1, \dots, S_u \subset \mathbb{N}^{(\leq m)}$ for $Y_1, \dots, Y_u$. Also, since $v \leq r/2 < 2^m$, we have relevant sets $S'_{u+1}, \dots S'_{r} \subset \mathbb{N}^{(\leq m-1)}$, and note that w.l.o.g. none of $S_1, S_2, \dots, S_u, S'_{u+1}, \dots, S'_r$ contains $x$. Set $S_{u+i} = S'_{u+i} \cup \{x\}$ for all $i = 1, \dots, v$. We claim that these have the desired property. So far, we know that for all $i$, $|S_i| \leq m$ holds.\\

Suppose that $\lambda_1, \dots, \lambda_r \in \mathbb{R}$ are such that $\sum_{j:S_i \subset X_j} \lambda_j = 0$ for all $i=1,2,\dots, r$. Define $\mu_i = \sum_{j: Y_i = X_{j} \setminus \{x\}} \lambda_j$, for each $i = 1, \dots, u$. Then we have $\sum_{j:S_i \subset Y_j} \mu_j = 0$ for all $i=1,2,\dots, u$. Since $n(Y_1, Y_2, \dots, Y_u; S_1, S_2, \dots, S_u) = 0$, we infer $\mu_j = 0$ for all $j$. Returning to the definion of $\mu_j$, we see that if $X_i$ is such that there no other $X_j$ with $X_i \setminus \{x\} = X_j \setminus \{x\}$, then $\lambda_i = 0$. On the other hand, if $i\not=j$ and $X_i \setminus \{x\} = X_j\setminus \{x\}$, then $\lambda_i = -\lambda_j$.\\

But, also setting $\nu_i = \lambda_j$ for $Z_i = X_j$, we have that for all $i=u+1, \dots, r$, $\sum_{j:S'_i \subset Z_j} \nu_j = 0$, which means that all $\nu_j=0$, as $n(Z_1, Z_2, \dots, Z_v; S'_{u+1}, S'_{u+2}, \dots, S'_r) = 0$. Combining these two conclusions, we have that all $\lambda_i = 0$, as desired.\\

\emph{Part (ii).} We follow the similar steps as in the previous part. However, we have to be slightly careful, since the previous argument unchanged would give us $K(X_1, X_2, \dots, X_r; S_1, S_2, \dots, S_r)$ essentially as a sum of kernels of similar matrices for $Y_1, Y_2, \dots, Y_u$ and $Z_{u+1}, Z_{u+2}, \dots, Z_r$. This way, we could be 1 dimension short of the desired goal, as this argument only allows us to deduce $n(X_1, X_2, \dots, X_r; S_1, S_2, \dots, S_r) \leq 2$, so we have to be more efficient. In order to overcome this issue, we shall apply the part (i) of the lemma.\\

We prove the claim by induction on $\sum_{i=1}^r |X_i|$. If this is zero, then we have $r=1$ and $X_1 = \emptyset$, so just take $S_1 = \emptyset$.\\

Suppose that the lemma holds for smaller values of $\sum_{i=1}^r |X_i|$. Let $x \in \mathbb{N}$ be any element that is contained in at least one of the sets $X_i$. Denote by $\{Y_1, Y_2, \dots, Y_u\}$ the collection of sets given by $\{X_i \setminus \{x\}: i=1,\dots, r\}$, and further let $\{Z_1, \dots, Z_v\}$ be the set $\{X_i: x \notin X_i, X_i \cup \{x\} = X_j\text{ for some }j\}$. Thus $v \leq u$ and $u+v = r$. Pick the sets $S_1, S_2, \dots, S_u \in \mathbb{N}^{(\leq m)}$ such that $U = K(Y_1, Y_2, \dots, Y_u; S_1, S_2, \dots, S_u)$ is of minimum dimension. Further, pick the sets $S'_{u+1}, S'_{u+2}, \dots, S'_r \in \mathbb{N}^{(\leq m-1)}$ such that $V = K(Z_1, Z_2, \dots, Z_v; S'_{u+1}, S'_{u+2}, \dots, S'_{r})$ is of minimum dimension. Finally, set $S_{u+i} = S'_{u+i} \cup \{1\}$ for $i = 1, \dots, v$. All $S_i$ have size at most $m$.\\

By induction hypothesis, we have $\dim U \leq 1$ and, since $v \leq r/2 < 3.2^{m-1}$, by induction hypothesis we have $\dim V \leq 1$. However, we can make a saving of one dimension as promised. Suppose that $\dim U = \dim V = 1$. Then, by part (i), since $U,V$ are of minimal possible dimension, we must have $u \geq 2^{m+1}$ and $v \geq 2^m$, which is a contradiction as $u + v = r < 3.2^m$. Therefore, $\dim U + \dim V \leq 1$.\\
We may reorder $X_1, X_2, \dots, X_r$, if necessary, to have $Y_i = X_i \setminus \{x\}$, for $i = 1, 2, \dots, u$, and $Z_i = X_{u+i} \setminus \{x\}$ with $x \in X_{u + i}$ for $i= 1,2, \dots, v$. Furthermore, we may also assume that $Z_i = X_{u-v+i} \setminus \{x\}$ with $x \notin X_{u-v+i}$ for $i= 1,2, \dots, v$. Now proceed as in the part (i), with the argument modified to suit the new context of possibly non-trivial kernels. Suppose that $\lambda_1, \dots, \lambda_r \in \mathbb{R}$ are such that $\sum_{j:S_i \subset X_j} \lambda_j = 0$ for all $i=1,2,\dots, r$. Define $\mu_i = \sum_{j: Y_i = X_{j} \setminus \{x\}} \lambda_j$, for each $i = 1, \dots, u$, thus
\begin{dmath*}\mu_i = \begin{cases} 
      \hfill \lambda_i    \hfill & \text{ if }i \leq u-v \\
      \hfill \lambda_i + \lambda_{i + v} \hfill & \text{ if } u-v < i \leq u\\
  \end{cases}\end{dmath*}

Then we have $\sum_{j:S_i \subset Y_j} \mu_j = 0$ for all $i=1,2,\dots, u$. This thus gives $\mu \in U$.\\

Next, set $\nu_i = \lambda_j$ for $Z_i = X_j$, i.e. $\nu_i = \lambda_{u + i}$ for $i = 1, 2, \dots, v$. We have $\sum_{j:T'_i \subset Z_j} \nu_j = 0$ for all $i=u+1, \dots, r$,  which means that $\nu \in V$. Expressing the $\lambda_i$ in terms of $\mu_i$ and $\nu_i$ we have
\begin{dmath*}\lambda_i = \begin{cases}
\hfill \mu_i    \hfill & \text{ if }i \leq u-v \\
\hfill \mu_i - \nu_{i+v - u}    \hfill & \text{ if }u-v < i \leq u \\
\hfill \nu_{i - u} \hfill & \text{ if }u < i\\
\end{cases}\end{dmath*}
Since $\mu \in U$ and $\nu \in V$, we can express any given $\lambda \in K(X_1, X_2, \dots, X_r; S_1, S_2, \dots, S_r)$ as a sum of vectors in two supspaces of $\mathbb{R}^r$, isomorphic to $U$ and $V$, so $K(X_1, X_2, \dots, X_r; S_1, S_2, \dots, S_r)$ is a subset of at most 1-dimensional subpsace, as desired.\end{proof}

The following corollary just restates the lemma in the context of the non-empty sets, as it will be required later in the light of the Observation~\ref{setobs}.

\begin{corollary}\label{nonemptyAlgCor}Let $m, k \in \mathbb{N}$. Given any distinct non-empty sets $X_1, X_2, \dots,X_r \in \mathbb{N}^{(<\omega)}$, we can find non-empty sets $S_1, S_2, \dots, S_r \subset \mathbb{N}^{(\leq m)}$ which enjoy the following property.
\begin{description}
\item[(i)] $n(X_1, X_2, \dots, X_r; S_1, S_2, \dots, S_r) = 0$, provided $r < 2^{m+1} - 1$.
\item[(ii)] $n(X_1, X_2, \dots, X_r; S_1, S_2, \dots, S_r) \leq 1$, provided $r < 3.2^{m} - 1$. 
\end{description}\end{corollary}

\begin{proof} In both cases, we apply the Lemma~\ref{algLemmaSets} to the distinct sets $\emptyset, X_1, X_2, \dots, X_r$ to find sets $S_0, S_1, \dots, S_r$ of size at most $m$ such that
$$n(\emptyset, X_1, X_2, \dots, X_r; S_0, S_1, S_2, \dots, S_r) \leq q$$
where $q = 0$ if $r < 2^{m+1} - 1$, and $q = 1$ if $r < 3.2^m - 1$. We now show that, starting from 
$$n(\emptyset, X_1, X_2, \dots, X_r; S_0, S_1, S_2, \dots, S_r) \leq q$$
we can reorder sets $S_i$ so that 
$$n(X_1, X_2, \dots, X_r; S_1, S_2, \dots, S_r) \leq q$$
which finishes the proof.\\

Let $I$ be the matrix $I(\emptyset, X_1, X_2, \dots, X_r; S_0, S_1, S_2, \dots, S_r)$. By rank-nullity theorem, the rank of $I$ (which is also the column rank) is at least $r+1 - q$. If all the sets $S_i$ are non-empty, then the first column of $I$ is zero. Removing the first row from $I$, we get a matrix with column rank also $\geq r+1-q$, thus having row rank also $\geq r+1-q$. Remove the first row, the remaining matrix is $I(X_1, X_2, \dots, X_r; S_1, S_2, \dots, S_r)$ and it has row rank at least $r-q$. Thus its rank is at least $r-q$, so by rank-nullity theorem, $n(X_1, X_2, \dots, X_r; S_1, S_2, \dots, S_r) \leq q$ as desired.\\
On the other hand, if $S_0 = \emptyset$ (after reordering if necessary), remove the first row from $I$, to get a matrix with row rank at least $r-q$, and whose first column is zero. But removing the first column doesn't change the column rank, and we end with matrix $I(X_1, X_2, \dots, X_r; S_1, S_2, \dots, S_r)$ of column rank $\geq r-q$, which by rank-nullity theorem gives
$$n(X_1, X_2, \dots, X_r; S_1, S_2, \dots, S_r) \leq q$$
as desired.\end{proof}


The next corollary is tailored to the analysis of the $\mathcal{F}_{N,d,m}-$incident sets in the Corollary~\ref{unionsCor}.

\begin{corollary}\label{algCor2}Suppose that $X_1, X_2, \dots, X_r \in \mathbb{N}^{(<\omega)}$ are distinct, $t \leq r$ and $S_1, S_2, \dots, S_t \in \mathbb{N}^{(<\omega)}$ satisfy
$$n(X_1, X_2, \dots, X_t; S_1, S_2, \dots, S_t) = 0.$$ Provided $r < 3.2^m$, we can find $S_{t+1}, S_{t+2}, \dots, S_r \in \mathbb{N}^{(\leq m)}$ such that 
$$n(X_1, X_2, \dots, X_r; S_1, S_2, \dots, S_r) \leq 1.$$
If $r < 3.2^m-1$ and the sets $X_i$ are non-empty, then additionaly, sets $S_{t+1}, S_{t+2}, \dots, S_r$ can be chosen to be non-empty.
\end{corollary}

\begin{proof}Apply the Lemma~\ref{algLemmaSets} (ii), to get sets $T_1, T_2, \dots, T_r \in \mathbb{N}^{(\leq m)}$ such that $n(X_1, X_2, \dots, X_r; T_1, T_2, \dots, T_r) \leq 1$, or the Corollary~\ref{nonemptyAlgCor} if $r < 3.2^m-1$ and the sets $X_i$ are non-empty, to make the sets $T_i$ non-empty. Look at the $(t+r)\times r$ matrix $I(X_1, X_2, \dots, X_r; S_1, S_2, \dots, S_t, T_1, T_2, \dots, T_r)$. We shall remove $t$ rows from those corresponding to $T_1, T_2,\dots,T_r$ to get the desired matrix. The following row-removal lemma does this for us.

\begin{lemma}Suppose that $A$ is $r+t \times r$ matrix with the first $t$ rows linearly independent and $t \leq r$. Then we can remove $t$ rows from the last $r$ rows of $A$, so that the kernel of $A$ doesn't change.\end{lemma}
\begin{proof} If $I \subset [r+t]$, let $A_I$ stand for the matrix formed from rows of $A$ with indices in $I$. Starting from $I = [r+s]$, we shall remove an element greater than $t$ from $I$, so that at each step we have $\ker A_I = \ker A$.\\
Suppose that we have $I \subset [r+t]$ with $[t] \subset I$, but $|I| > r$, such that $\ker A_I = \ker A$ holds. If we can pick $x > t$ in $I$, so that $\ker A_{I\setminus \{x\}} = \ker A_I$, we are done. Otherwise, no such $x$ works. Observe that if a row $v^T$ of $A_I$ is a linear combination of other rows, then it can be removed from $A_I$. To spell it out, write $v_i^T$ for $i$-th row of $A$ and  suppose that $v_i^T = \sum_{j \in I \setminus \{i\}} \lambda_j v_j^T$. Then, if $\mu \in \ker A_{I \setminus \{i\}}$, we have $\mu \cdot v_i^T = \sum_{j \in I \setminus \{i\}} \lambda_j \mu\cdot v_j^T  = 0$. So $\ker A_{I \setminus \{x\}} = \ker A_I$.\\
Thus, we have $v_1^T, \dots, v_t^T$ linearly independent, and $v_i \notin span \{v_j: j \in I \setminus \{i\}\}$ for $i \in I \setminus [t]$. But, then, $|I| > r$ and the rows of $I$ are linearly independent, but are of length $r$, which is contradiction. Hence, we can proceed, until we reach $|I| = r$, as desired.\end{proof}

The matrix
$$I(X_1, X_2, \dots, X_r; S_1, S_2, \dots, S_t, T_1, T_2, \dots, T_r)$$
satisfies the conditions of the lemma since $n(X_1, X_2, \dots, X_t; S_1, S_2, \dots, S_t) = 0$, so by applying the lemma, we can pick $S_{t+1}, S_{t+2}, \dots, S_r$ among the sets in $T_1, T_2, \dots, T_r$ so that
\begin{dmath*}n(X_1, X_2, \dots, X_r; S_1, S_2, \dots, S_r) = n(X_1, X_2, \dots, X_r; S_1, S_2, \dots, S_t, T_1, T_2, \dots, T_r) \leq n(X_1, X_2, \dots, X_r; T_1, T_2, \dots, T_r) \leq 1.\end{dmath*}
\end{proof}

Similarly to the Corollary~\ref{minincline}, the next corollary is consistent with the vague idea that combinatorial subspaces are the source of non-trivial $\mathcal{F}_{N,d,m}-$incident sets. In particular, we show that $\mathcal{F}_{N,d,m}-$incident sets behave like combinatorial subspaces when it comes to taking unions -- the size of union of two $\mathcal{F}_{N,d,m}-$incident sets of size $2^{m+1}$ is at least $3.2^m$. 

\begin{corollary}\label{unionsCor}Let $d,m \in \mathbb{N}$ be given. 
\begin{description}
\item[(i)] If $T \subset \{0,1\}^N$ is $\mathcal{F}_{N,d,m}-$incident, then $|T| \geq \min\{d+2, 2^{m+1}\}$.
\item[(ii)] If $T_1, T_2 \subset \{0,1\}^N$ are distinct, of size at most $d+1$ and minimal (w.r.t. inclusion) $\mathcal{F}_{N,d,m}-$incident, then $|T_1 \cup T_2| \geq 3.2^m.$
\end{description} 
\end{corollary}
\begin{proof} \emph{Part (i).} Suppose that $T = \{x_0, x_1, x_2, \dots, x_r\} \subset \{0,1\}^N$ is $\mathcal{F}_{N,d,m}$-incident and that $r < 2^{m+1}-1, d+1$. Note that the map $X \mapsto X \Delta A$, corresponds to a reflection $\alpha:\mathbb{R}^N \to \mathbb{R}^N$. In particular, taking $A$ to be the set of non-zero coordinates of $x_0$, we have an affine isomorphism $\alpha$ that preserves the cube $\{0,1\}^N$ and sends $x_0$ to zero. Let $X_i\subset [N]$ be the set corresponding to $\alpha(x_i)$, i.e. the set of indices $j$ such that $\alpha(x_i)_j = 1$. As $r < 2^{m+1}-1$, the Corollary~\ref{nonemptyAlgCor} yields non-empty sets $S_1, S_2, \dots, S_r \subset [N]$ of size at most $m$, such that 
$$n(X_1, X_2, \dots, X_r; S_1, S_2, \dots, S_r) = 0.$$
Choosing $(\leq m)-$functions $f_1, f_2, \dots, f_r$ with images $S_1, S_2, \dots, S_r$ we obtain that the vectors
\begin{dmath*}\begin{pmatrix} 0\\0\\\vdots\\0\end{pmatrix}, \begin{pmatrix} f_1(\alpha(x_1))\\f_2(\alpha(x_1))\\\vdots\\f_r(\alpha(x_1))\end{pmatrix}, \begin{pmatrix} f_1(\alpha(x_2))\\f_2(\alpha(x_2))\\\vdots\\f_r(\alpha(x_2))\end{pmatrix}, \dots, \begin{pmatrix} f_1(\alpha(x_r))\\f_2(\alpha(x_r))\\\vdots\\f_r(\alpha(x_r))\end{pmatrix}\end{dmath*}
are affinely independent. But, as $r \leq d$, the Proposition~\ref{incCriterion} applies to $T$, affine map $\alpha$ and functions $f_1, f_2, \dots, f_r$, which tells us that these vectors are affinely dependent, which is a contradiction. Thus $|T| = r + 1 \geq \min\{2^{m+1}, d+2\}$ as desired.\\

\emph{Part (ii).} If $T_1, T_2$ are disjoint, then by part (i), $|T_1 \cup T_2| \geq 2^{m+2}$, so we are done. Thus, assume that some $t_0$ belongs to both sets. Pick an affine isomorphism $\alpha:\mathbb{R}^N \to \mathbb{R}^N$ which sends $t_0$ to zero and preserves the cube $\{0,1\}^N$ (given by a suitable reflection). Let $X_1, X_2, \dots, X_t$ be the sets that correspond to the non-zero points of $\alpha(T_1 \cap T_2)$, $X_{t+1}, \dots, X_{t+r_1}$ be the sets that correspond to points in $\alpha(T_1 \setminus T_2)$ and $X_{t+r_1+1}, \dots, X_{t+r_1 + r_2}$ be the sets corresponding to points of $\alpha(T_2 \setminus T_1)$. If $|T_1 \cup T_2| \geq 3.2^m$, we are done. Otherwise $1 + t+r_1 +r_2 = |T_1 \cup T_2| < 3.2^m$.\\ 
Since they are minimal and distinct, $T_1, T_2$ cannot contain one another. So $T_1 \cap T_2$ is a proper subset of $T_1$ and hence is not $\mathcal{F}_{N,d,m}-$incident. Therefore, by the Observation~\ref{setobs}, we can find non-empty $S_1, S_2, \dots, S_t \in \mathbb{N}^{(\leq m)}$ such that
$$n(X_1, X_2, \dots, X_t; S_1, S_2, \dots, S_t) = 0.$$
Applying the Corollary~\ref{algCor2} (as $r+t_1+t_2 <3.2^m-1$), we obtain non-empty sets $S_{t+1}, \dots, S_{t+r_1+r_2} \in \mathbb{N}^{(\leq m)}$ such that
$$n(X_1, X_2, \dots, X_{r+t_1+t_2}, S_1, S_2, \dots, S_{r+t_1 +t_2}) \leq 1.$$
Now, take any $(\leq m)$-functions $f_1, \dots, f_{t+r_1 +r_2}$ to $[N]$ with images $S_1, S_2, \dots, S_{t+r_1+r_2}$, and let $x_i \in T_1 \cup T_2$ be point such that $X_i$ corresponds to $\alpha(x_i)$. Write $y_i$ for the vector $y_{ij} = f_j(x_i)$, $j=1,2,\dots, {t+r_1+r_2}$. Thus, $y_1, y_2, \dots, y_t$ are linearly independent and the rank of $y_1, y_2, \dots, y_{t + r_1 + r_2}$ is at least $t+r_1 +r_2 - 1$. Since $|T_1| \leq d+1$, we can apply the Proposition~\ref{incCriterion} to $T_1$, map $\alpha$ and functions $f_1, \dots, f_{t+r_1}$. Note that since the sets $S_i$ are non-empty, we have $f_i(0) = 0$ for all $i$. Thus, vectors $y_1, y_2, \dots, y_{t+r_1}$ have rank at most $t+r_1 -1$. Similarly, rank of $y_1, y_2, \dots, y_t, y_{t+r_1+1}, \dots, y_{t+r_1+r_2}$ is at most $t+ r_2 -1$.\\
To obtain contradiction, look at
\begin{description}
\item $U = span \{y_1, y_2, \dots, y_{t+r_1}\},$
\item $V = span \{y_1, y_2, \dots, y_{t}, y_{t+r_1+1}, y_{t+r_1+2}, \dots, y_{t+r_1+r_2}\},$
\item $W = span \{y_1, y_2, \dots, y_{t+r_1+r_2}\}$ and
\item $Z = span \{y_1, y_2, \dots, y_{t}\}.$
\end{description}
Thus, $\dim Z = t, \dim U \leq t+r_1-1, \dim V \leq t+r_2-1, \dim W \geq t+r_1+r_2-1$, $Z \subset U, V \subset W$ and $W = U+V$. Therefore $W/Z = U/Z + V/Z$. Finally, $r_1 +r_2 - 1 \leq \dim W - \dim Z = \dim W/Z \leq \dim U/Z + \dim V/Z \leq r_1 - 1 + r_2 - 1 = r_1 + r_2 - 2$, which is contradiction.\end{proof}

\begin{theorem}Suppose that $d,m \in \mathbb{N}$ satisfy $2^{m+1} - 1 \leq d \leq 3.2^m - 3$. Let $N \geq 1$. Then $$\alpha(2^N,d) \leq \left(\frac{25}{N}\right)^{1/2^{m+1}} 2^N.$$\end{theorem}

\begin{proof}Let $X = \{0,1\}^N \subset \mathbb{R}^N$. Applying the Proposition~\ref{incremfn}, we obtain a function $f \in span \mathcal{F}_{N, d, m}$, bijection onto its image when restricted to $X$, such that if $S \subset f(X)$ is affinely dependent then $f^{-1}(S)$ is $\mathcal{F}_{N,d,m}$-incident. Let $Y = f(X)$. Note that $|Y| = 2^N$ since $f$ is injective on $X$. We claim that $Y$ has no more than $d+1$ points on a same hyperplane, but all sufficiently large subsets of $Y$ have $d+1$ cohyperplanar points.\\

\emph{No more than $d+1$ points on a hyperplane.} Suppose that we have a set $S = \{s_1, s_2, \dots, s_{d+2}\} \subset Y$ that is a subset of a hyperplane. Pick a maximal affinely independent subset $S' \subset S$. W.l.o.g. $S' = \{s_1, \dots, s_r\}$, for some $r$. As $S'$ is a subset of a hyperplane, we have $r \leq d$. Look at $S'_1 = S' \cup \{s_{d+1}\}$. By the choice of $S'$, the set $S'_1$ is not affinely independent. By the choice of $f$, the preimage $f^{-1}(S'_1)$ is $\mathcal{F}_{N,d,m}-$incident. Find a subset $T_1$ of $f^{-1}(S'_1)$ which is minimal $\mathcal{F}_{N,d,m}-$incident, and arbitrary point $p$ in $T_1$. We also have $S'_2 = S \setminus \{p\}$ affinely dependent, as it is a subset of a hyperplane of size at $d+1$. By the choice of $f$, $f^{-1}(S'_2)$ is $\mathcal{F}_{N,d,m}-$incident, and has a minimal $\mathcal{F}_{N,d,m}-$incident subset $T_2$. Note that $p \in T_1 \setminus T_2$, so $T_1, T_2$ are distinct, and $|T_1|, |T_2| \leq d+1$. The Corollary~\ref{unionsCor}(ii) applies to give $d+2 = |S| \geq |T_1 \cup T_2| \geq 3.2^m > d+2$, which is a contradiction.

\emph{Dense subsets are not in general position.} Let $T \subset Y$ have size at least $\left(\frac{25}{N}\right)^{1/2^{m+1}} 2^N$. Then, by the Theorem~\ref{sperner}, $f^{-1}(T)$ contains a $m+1$-dimensional combinatorial subspace. Applying the Lemma~\ref{exInc}, we have that the points of $T = f(f^{-1}(T))$ are affinely dependent. Adding any $d+1 - 2^{m+1}$ points to the set $T$ proves the claim.\end{proof}

\section{Conclusion}
Even though there are now some non-trivial estimates of $\alpha(n,d)$~\cite{CardinalTothWood}, the gap between the lower and upper bounds is still very large. Of course, the first question is still to determine the $\alpha(n,d)$. Regarding the current lower bounds on $\alpha(n,2)$, both in~\cite{Furedi} and in~\cite{CardinalTothWood}, we note that their proofs are based on relatively general probabilistic estimates of independence number of hypergraphs. However, these approaches used very little of the structure the given sets of points. In fact, possible algebraic properties of such sets have not been exploited. For example, if $X$ is a set of points with no more than 3 on a line, but with no dense set in general position, we can expect that plenty of pairs of points in $X$ have a third point in $X$ on their line. This gives raise to an algebraic operation: given two points $x_1, x_2$ of $X$, set $x_1 \ast x_2$ to be the third point of $X$ on their line, if such a point exists. Of course, there is an issue of how to define $x_1 \ast x_2$ for all pairs, but at least for plenty of pairs it can be defined. Hopefully, if $X$ is a set for which the $\alpha(|X|, 2)$ is attained, we could deduce some properties of the operation $\ast$.\\   

Leaving determination of $\alpha(n,d)$ aside, note that it would be surprising if the $\alpha(n,d)/n$ did not decrease in $d$. In particular, the current situation with the upper bounds is that we have infinitely many $d$, for which $\alpha(n,d)/n = O(1/log^{\beta_d} n)$ for some $\beta_d > 0$, while for infintely many other $d$, the bounds for $\alpha(n,d)/n$ are coming from the density Hales-Jewett theorem, and are rougly comparable to inverse of Ackermann function. It is most probably far from truth that $\alpha(n,d)/n$ is actually close to these estimates. An obvious question is the following.

\begin{question}What is the relationship between $\alpha(n,d_1)$ and $\alpha(n,d_2)$ for $d_1 < d_2$? Do we always have $\alpha(n,d)/n \geq \alpha(n,d+1)/n$?\end{question}     

Another question is the relationship between the bounds in the density Hales-Jewett theorem and the $\alpha(n,d)$. Do these have to be related?

Finally, one of the key tools in this paper were the algebraic lemmas~\ref{algebraicLemma1} and~\ref{algLemmaSets}. It could be of interest to study $n(X_1, X_2, \dots, X_r; S_1, S_2, \dots, S_r)$ further.



\section{Acknowledgments}
I would like to thank Trinity College and the Department of Pure Mathematics and Mathematical Statistics of Cambridge University for their generous support and Imre Leader for the helpful discussions concerning this paper.

\end{document}